\documentclass[11pt,reqno]{amsart}

\numberwithin{equation}{section}

\usepackage{color,graphics,epsfig,amsmath}

\newcommand{\calA}{\mathcal{A}}

\newcommand{\calG}{\mathcal{G}}

\newcommand{\calL}{\mathcal{L}}

\newcommand{\calT}{\mathcal{T}}

\newcommand{\mC}{\mathbb{C}}

\newcommand{\mF}{\mathbb{F}}

\newcommand{\mR}{\mathbb{R}}
\newcommand{\mS}{\mathbb{S}}

\newcommand{\mZ}{\mathbb{Z}}

\newcommand{\inv}{{\textrm{inv }}}

\newtheorem{theorem}{Theorem}[section]
\newtheorem{lemma}[theorem]{Lemma}

\newtheorem{proposition}[theorem]{Proposition}

\theoremstyle{definition}

\newtheorem{remark}[theorem]{Remark}

\theoremstyle{definition}
\newtheorem{definition}[theorem]{Definition}

\theoremstyle{definition}

\begin{document}



\title[$\nu$-metric versus gap-metric]{The new $\nu$-metric induces
  the classical gap topology}

\author{Amol Sasane}
\address{Department of Mathematics, Royal Institute of Technology,
    Stockholm, Sweden.}
\email{sasane@math.kth.se}

\begin{abstract}
  Let $\calA_+$ denote the set of Laplace transforms of complex Borel
  measures $\mu$ on $[0,+\infty)$ such that $\mu$ does not have a
  singular non-atomic part. In \cite{BalSas}, an extension of the
  classical $\nu$-metric of Vinnicombe was given, which allowed one to
  address robust stabilization problems for unstable plants over
  $\calA_+$. In this article, we show that this new $\nu$-metric gives
  a topology on unstable plants which coincides with the classical gap
  topology for unstable plants over $\calA_+$ with a single input and
  a single output.
\end{abstract}

\maketitle

\section{Introduction}

We recall the general {\em stabilization problem} in control theory.
Suppose that $R$ is a commutative integral domain with identity
(thought of as the class of stable transfer functions) and let
$\mF(R)$ denote the field of fractions of $R$. Then the stabilization
problem is:

\medskip

\begin{center}
\parbox[r]{11cm}{Given $p\in \mF(R) $ (an unstable plant transfer function),

 find $c \in \mF(R)$ (a stabilizing controller
  transfer function),

 such that (the closed loop transfer function)
$$
H(p,c):= \left[\begin{array}{cc} p \\ 1 \end{array} \right]
(1-cp)^{-1} \left[\begin{array}{cc} -c & 1 \end{array} \right]
$$
belongs to $R^{2\times 2}$ (that is, it is stable).}
\end{center}

\medskip

\noindent In the {\em robust stabilization problem}, one goes a step
further.  One knows that the plant is just an approximation of
reality, and so one would really like the controller $c$ to not only
stabilize the {\em nominal} plant $p$, but also all sufficiently close
plants $p'$ to $p$.  The question of what one means by ``closeness''
of plants thus arises naturally. So one needs a function $d$ defined
on pairs of stabilizable plants such that
\begin{enumerate}
\item $d$ is a metric on the set of all stabilizable
plants,
\item $d$ is amenable to computation, and
\item stabilizability is a robust property of the plant with respect
  to $d$.
\end{enumerate}
Such a desirable metric, was introduced by Glenn Vinnicombe in
\cite{Vin} and is called the $\nu$-{\em metric}. In that paper,
essentially $R$ was taken to be the rational functions without poles
in the closed unit disk, and it was also shown that the topology
obtained was equivalent to the one obtained from the gap-metric
(introduced by Zames and El-Sakkary \cite{ZamElS},\cite{ElS}, which in
turn is equivalent to the graph metric of Vidyasagar \cite{Vid}).
 
The problem of what happens when $R$ is some other ring of stable
transfer functions of infinite-dimensional systems was left open in
\cite{Vin}. This problem of extending the $\nu$-metric from the
rational case to transfer function classes of infinite-dimensional
systems was addressed in \cite{BalSas}. There the starting point in
the approach was abstract. It was assumed that $R$ is any commutative
integral domain with identity which is a subset of a Banach algebra
$S$ satisfying certain assumptions, and then an ``abstract''
$\nu$-metric was defined in this setup, and it was shown in
\cite{BalSas} that it does define a metric on the class of all
stabilizable plants.  It was also shown there that stabilizability is
a robust property of the plant. In particular, this gave a metric on
unstable plants over $\calA_+$, where $\calA_+$ denotes the set of
Laplace transforms of complex Borel measures $\mu$ on $[0,+\infty)$
such that $\mu$ does not have a singular non-atomic part.

One can also define a gap-metric for unstable plants over $\calA_+$,
and so it is natural to ask if the $\nu$-metric and the gap-metric
induce the same topologies on unstable plants over $\calA_+$. In this
article we address this issue, and prove the following result.

\begin{theorem}
\label{main_theorem}
  On the set $\mS(\calA_+)$, the topologies induced by the
  $\nu$-metric $d_\nu$ and the gap-metric $d_g$ are identical.
\end{theorem}

The notation $\mS(\calA_+)$ will be explained carefully in
Section~\ref{section_abstract_nu_metric}, but roughly speaking, it is
to be thought of as the class of unstable plants over $\calA_+$ with a
single input and a single output. Owing to a technical difficulty, we
restrict ourselves to single input and single output systems. We end
this article with an open problem, namely the validity of our main
result for systems with multiple inputs and multiple outputs, while
pointing out the precise nature of the technical difficulty.

The paper is organized as follows:
\begin{enumerate}
\item In Section~\ref{section_abstract_nu_metric}, we recall from
  \cite{BalSas} the $\nu$-metric in the context of unstable plants
  over $\calA_+$, and also derive an alternative expression for it in
  Proposition~\ref{prop_alt_exp_nu_metric}, reminiscent of Georgiou's
  formula for the gap-metric from \cite{Geo}.
\item In Section~\ref{section_gap_metric}, we give the definition of
  the gap-metric in the context of unstable plants over $\calA_+$. An
  alternative expression for the gap-metric is given in
  Proposition~\ref{prop_alt_exp_gap_metric}, which will be used in
  order to show the equivalence of $d_\nu$ and $d_g$.
\item Finally, in Section~\ref{section_equivalence}, we will prove our
  main result (Theorem~\ref{main_theorem}). At the end of this
  section, we also highlight the main obstacle towards extending
  Theorem~\ref{main_theorem} to systems with multiple inputs and outputs.
\end{enumerate}

\section{Notation index}

For the convenience of the reader, we have included a table here which
shows the page numbers of the places where the corresponding notation
is first defined.

\medskip

\begin{center}
\begin{tabular}{|c||l|} \hline
  Notation              & Page number \\ \hline \hline
  $\widehat{\;\cdot\;}$ & Laplace transform (page \pageref{pageref_Laplace}) or \\
                        & Fourier transform (page \pageref{pageref_Fourier}) \\ \hline
  ${\cdot}^*$           & pages \pageref{pageref_star_first}, \pageref{pageref_star_a}, 
                          \pageref{pageref_star_again}  \\ \hline
  $\calA$               & page \pageref{pageref_calA} \\ \hline
  $\calA_+$             & page \pageref{pageref_calA_+} \\ \hline 
  $AP$                  & almost periodic functions (page \pageref{pageref_AP}) \\ \hline 
  $C_0$                 & functions vanishing at $\pm\infty$  (page \pageref{pageref_C_0})\\ \hline 
  $\mC_+$               & right half of the complex plane (page \pageref{pageref_mC+})\\ \hline 
  $\;\;\vec{\delta}\phantom{\widehat{f^f}}$        & directed gap  (page \pageref{pageref_vec_delta})\\ \hline 
  $\;\;\;\;d_g\phantom{{g_{p}}_p}$                 & gap-metric    (page \pageref{pageref_d_g})                   \\ \hline
  $d_\nu$               & $\nu$-metric (page \pageref{eq_nu_metric}) \\ \hline
  $\mF(\calA_+)$        & field of fractions over $\calA_+$ (page \pageref{subsec1}) \\ \hline
  $\calG$               & graph of a system (page \pageref{pageref_calG}) \\ \hline
  $\;\;G, \widetilde{G}, K, \widetilde{K}\phantom{\widehat{f^f}}$ 
                        & matrices built from coprime factorizations (page \pageref{subsec5})\\ \hline
  $\inv \cdot$          & invertible elements of a ring (page \pageref{pageref_inv})\\ \hline
  $P_\calG$              & projection onto $\calG$ (page \pageref{pageref_P_calG}) \\ \hline 
  $P_{\calG_1}|_{\calG_2}$ & restriction of $P_{\calG_1}$ to $\calG_2$ (page \pageref{pageref_P_calG_1_calG_2}) \\ \hline 
  $\mS(\calA_+)$        & plants with a normalized coprime factorization (page \pageref{subsec6})\\ \hline 
  $T_X$                 & Toeplitz operator (page \pageref{pageref_T_X}) \\ \hline 
  ${\tt w}$             & winding number for continuous closed\\
                        & curves avoiding $0$ (page \pageref{pageref_wind_no})\\ \hline
  $w$                   & average winding number for invertible\\
                        & $AP$ functions (page \pageref{pageref_av_wind_no})\\ \hline
  $W$                   & index for invertible elements in $\calA$ (page \pageref{pageref_index_W}) \\ \hline
\end{tabular}

\end{center}

\section{The $\nu$-metric}
\label{section_abstract_nu_metric}

In this section we will recall the new $\nu$-metric for unstable
plants over the ring $\calA_+$ (defined below), which was listed as a
particular example in \cite[Subsection~5.3]{BalSas} of the abstract
$\nu$-metric introduced in that paper. At the end of this section, we
will also give an alternate expression for the $\nu$-metric, which
will be used later in order to show the equivalence of the
$\nu$-metric topology with the classical gap topology.

If $R$ is a commutative integral domain with identity $1$, we use the
symbol $\inv R$\label{pageref_inv} for the set of invertible elements of $R$.

We denote by $\calA_+$ the set of Laplace transforms of complex Borel
measures $\mu$ on $[0,+\infty)$ such that $\mu$ does not have a
singular non-atomic part. A more explicit description of the elements
of $\calA_+$ can be given as follows. Let \label{pageref_mC+}
$$
\mC_{+}:=\{s\in \mC: \textrm{Re}(s)\geq 0\}.
$$
Then \label{pageref_calA_+}
$$
\calA_+ = \bigg\{  s (\in \mC_{+}) \mapsto \widehat{f_a}(s)
  +\displaystyle \sum_{k\geq 0} f_k e^{- s t_k}  \Big| 
\begin{array}{ll}
 f_a \in L^{1}(0,\infty), \;(f_k)_{k\geq 0} \in \ell^{1},\\
 0=t_0 <t_1 ,t_2 , t_3, \dots
 \end{array} 
\bigg\},
$$
and equipped with pointwise operations and the norm:
$$
\|F\|=\|f_a\|_{\scriptscriptstyle L^{1}} +
\|(f_k)_{k\geq 0}\|_{\scriptscriptstyle \ell^1},
\;\; F(s)=\widehat{f_a}(s) +\displaystyle\sum_{k\geq 0} f_k e^{-st_k}\;\;(s\in \mC_+),
$$
$\calA_+$ is a Banach algebra. Here $\widehat{f_a}$ denotes the {\em
  Laplace transform} of $f_a$:\label{pageref_Laplace}
$$
\widehat{f_a}(s)=\displaystyle \int_0^\infty e^{-st} f_a(t)
dt, \quad s \in \mC_+.
$$
Similarly, define $\calA$ as follows:\label{pageref_calA}
$$
\!\!\!\!\!\!\calA\!=\!\bigg\{ iy (\in i\mR) \mapsto \widehat{f_a}(iy)
  +\!\!\!\displaystyle\sum_{k\in \mZ} f_k e^{- iy t_k} \Big|
\begin{array}{ll}
f_a \in L^{1}(\mR), \;(f_k)_{k\in \mZ } \in \ell^{1},\\
\dots, t_{-2}, t_{-1}<\!0\!=\!t_0\! <t_1 ,t_2 ,  \dots
\end{array} \!\!\!\bigg\}.
$$
Then, equipped with pointwise operations and the norm:
$$
\|F\|=\|f_a\|_{\scriptscriptstyle L^{1}} + \|(f_k)_{k\in
  \mZ}\|_{\scriptscriptstyle \ell^1}, \;\; F(iy):=\widehat{f_a}(iy)
+\displaystyle\sum_{k\in \mZ} f_k e^{-iy t_k}\;\;(y\in \mR),
$$
$\calA$ is a unital commutative complex semisimple Banach algebra.
Here $\widehat{f_a}$ is the {\em Fourier transform} of $f_a$, \label{pageref_Fourier}
$$
\widehat{f_a}(iy)= \displaystyle \int_{-\infty}^\infty e^{-iyt} f_a(t)
dt \quad (y \in \mR).
$$
One can also define an involution $\cdot^*$ on $\calA$, given by \label{pageref_star_first}
$$
F^*(iy)=\overline{F(iy)}, \quad y \in \mR,
$$
for $F\in \calA$. Clearly, $\calA_+ \subset \calA$.

The algebra $AP$ of complex valued (uniformly) {\em almost periodic
  functions}\label{pageref_AP} is the smallest closed subalgebra of $L^\infty(\mR)$ that
contains all the functions $e_\lambda := e^{i \lambda y}$. Here the
parameter $\lambda$ belongs to $\mR$.  For any $f\in AP$, its {\em
  Bohr-Fourier series} is defined by the formal sum
\begin{equation}
\label{eq_BFs}
\sum_{\lambda} f_\lambda e^{i  \lambda y} , \quad y\in \mR,
\end{equation}
where
$$
f_\lambda:= \lim_{N\rightarrow \infty} \frac{1}{2N}
\int_{[-N,N]}   e^{-i \lambda y} f(y)dy, \quad
\lambda \in \mR,
$$
and the sum in \eqref{eq_BFs} is taken over the set $
\sigma(f):=\{\lambda \in \mR\;|\; f_\lambda \neq 0\}$, called the {\em
  Bohr-Fourier spectrum} of $f$. The Bohr-Fourier spectrum of every
$f\in AP$ is at most a countable set. For each $f\in \inv AP$, we can
define the {\em average winding number}\label{pageref_av_wind_no}
$w(f)\in \mR$ of $f$ as follows \cite[Theorem 1, p. 167]{JesTor}:
$$
w(f)= \lim_{T \rightarrow \infty} \frac{1}{2T}
\Big( \arg (f(T))-\arg(f(-T))\Big).
$$
We set 
$$
F_{AP}(iy)=\displaystyle\sum_{k\in \mZ} f_k e^{-iy
  t_k}\;\;(y\in \mR) \quad \textrm{for} \quad F=\widehat{f_a}
+\displaystyle\sum_{k\in \mZ} f_k e^{-i\cdot t_k}\in \calA.
$$
If $F =\widehat{f_a}+F_{AP} \in \inv \calA$, then it can be shown that
(\cite[Subsection~5.3]{BalSas}) $F_{AP}(i \cdot) \in \inv AP$.
Moreover, $F=\widehat{f_a}+F_{AP} \in \calA$ is invertible if and only
if for all $y\in \mR$, $F(iy) \neq 0$ and $\displaystyle \inf_{y\in
  \mR} |F_{AP}(iy)| >0$.

Since $\widehat{L^1(\mR)}$ is an ideal in $\calA$, it follows that
$F_{AP}^{-1}\widehat{f_a}$ is the Fourier transform of a function in
$L^{1}(\mR)$, and so the map
$$
y \mapsto
1+(F_{AP}(iy))^{-1}\widehat{f_a}(iy)=\frac{F(iy)}{F_{AP}(iy)}
$$
has a well-defined winding number ${\tt w}$ \label{pageref_wind_no}
around $0$. Geometrically, ${\tt w}(f)$ is the number of times the
curve $t \mapsto f(t)$ winds around the origin in a counterclockwise
direction.
 
Define the {\em index} $W: \inv \calA \rightarrow \mR\times \mZ$ by
\begin{equation}   \label{pageref_index_W}
 W (F)= \Big(w(F_{AP}), {\tt
  w}(1+F_{AP}^{-1} \widehat{f_a}) \Big),
\end{equation}
 where
$F=\widehat{f_a}+F_{AP} \in \inv \calA$, and
$$
\begin{array}{ll}
w(F_{AP})
:=
\displaystyle \lim_{R \rightarrow \infty} \frac{1}{2R}
\Big( \arg \big(F_{AP}(iR)\big)-\arg\big(F_{AP}(-iR)\big)\Big),
\\
{\tt w}(1+F_{AP}^{-1} \widehat{f_a})
:=
\displaystyle \frac{1}{2\pi}\Big( \arg \big(1+(F_{AP}(iy)\big)^{-1} \widehat{f_a}(iy)  )
\Big|_{y=-\infty}^{y=+\infty}\Big).\phantom{ \frac{1^f}{2\pi}}
\end{array}
$$
The map $W: \inv \calA \rightarrow \mR\times \mZ$ satisfies:
\begin{itemize}
\item[(I1)] $W(ab)= W (a) +W(b)$ ($a,b \in \inv \calA$).
\item[(I2)] $W(a^*)=-W(a)$ ($a\in \inv \calA$).
\item[(I3)] $W$ is locally constant, that is, $W$ continuous
  when $\mR\times \mZ$ is equipped with the discrete topology.
\item[(I4)] $x\in \calA_+ \cap (\inv \calA)$ is invertible as an
  element of $\calA_+$ if and only if $W(x)=(0,0)$.
\end{itemize}
A consequence of (I3) is the following ``homotopic invariance of the
index'' (see \cite[Proposition~2.1]{BalSas}): if $H:[0,1] \rightarrow
\inv \calA$ is a continuous map, then $W (H(0))=W(H(1))$.

We recall the following standard notation and definitions from the factorization
approach to control theory.

\subsection{The notation $\mF(\calA_+)$:}
\label{subsec1} $\mF(\calA_+)$ denotes the field of  fractions of $\calA_+$.

\subsection{The notation $F^*$:}
\label{pageref_star_a} If $F\in \calA_+^{p\times m}$, then
$F^*\in \calA^{m\times p}$ is the matrix with the entry in the $i$th
row and $j$th column given by $F_{ji}^*$, for all $1\leq i\leq p$, and
all $ 1\leq j \leq m$.

\subsection{Coprime/normalized coprime factorization:} Given $p \in
\mF(R)$, a factorization $p=nd^{-1}$, where $n,d \in R$, is called a
{\em coprime factorization of} $P$ if there exist $x, y \in R$ such
that $ x n + y d=1$.  If moreover there holds that $ n^{*} n +d^{*} d
=1$, then the coprime factorization is referred to as a {\em
  normalized} coprime factorization of $p$.

\subsection{The notation $G, \widetilde{G}, K,\widetilde{K}$:}  
\label{subsec5}
Given $p \in \mF(\calA_+)$ with a normalized coprime factorization
$p=nd^{-1}$, we introduce the following matrices with entries from
$\calA_+$:
$$
G=\left[ \begin{array}{cc} n \\ d \end{array} \right] \quad
\textrm{and} \quad \widetilde{G}=\left[ \begin{array}{cc} -d & n
    \end{array} \right] .
$$
Similarly, given $c \in \mF(\calA_+)$ with normalized coprime
factorization $c=xy^{-1}$, we introduce the following matrices with
entries from $\calA_+$:
$$
K=\left[ \begin{array}{cc} y \\ x \end{array} \right] \quad
\textrm{and} \quad \widetilde{K}=\left[ \begin{array}{cc} -x &
    y\end{array} \right] .
$$

\subsection{The notation $\mS(\calA_+)$:}
\label{subsec6} 
We denote by $\mS(\calA_+)$ the set of all elements $p\in
\mF(\calA_+)$ that possess a normalized coprime factorization. 

\begin{remark}\label{remark}
$\;$
\begin{enumerate}
\item It can be shown (see for example \cite[Chapter 8]{Vid}) that if $ p\in
\mS(\calA_+)$, then $p$ is a {\em stabilizable plant over} $\calA_+$,
that is, there exists a $c\in \mF(\calA_+)$ such that $H(p,c)\in
R^{2\times 2}$.
\item \cite[Subsection~3.5]{BruSas} shows that every
stabilizable plant $p\in \mF(\calA_+)$ admits a coprime factorization
over $\calA_+$. 
\item 
\label{Kalle_remark}
It follows from the proof of \cite[Lemma~6.5.6.(e)]{Mik} and
\cite[Theorem~5.2.8]{Mik} that whenever $p\in \mF(\calA_+)$ has a
coprime factorization over $\calA_+$, it also has a {\em normalized}
coprime factorization over $\calA_+$.
\end{enumerate}
\end{remark}
Putting these remarks together, we see that $\mS(\calA_+)$ is exactly
the set of all plants in $\mF(\calA_+)$ that are stabilizable over
$\calA_+$.

\begin{definition}[$\nu$-metric $d_\nu$ on $\mS(\calA_+)$]
\label{def_nu_metric}
For $p_1, p_2 \in \mS(\calA_+)$, with the normalized coprime
factorizations $p_1= n_{1} d_{1}^{-1}$ and $p_2= n_{2} d_{2}^{-1}$, 
we define
\begin{equation}
\label{eq_nu_metric}
d_{\nu} (p_1,p_2 ):=\left\{
\begin{array}{ll}
  \|\widetilde{G}_{2} G_{1}\|_\infty &
  \textrm{if } G_1^* G_2 \in \inv \calA \textrm{ and }
  W (G_1^* G_2)=(0,0), \\
  1 & \textrm{otherwise}. \end{array}
\right.
\end{equation}
where the notation is as in Subsections~\ref{subsec1}-\ref{subsec6}.
\end{definition}

We have the following; see \cite{BalSas}:

\begin{theorem}
\label{thm_d_nu_is_a_metric}
$d_\nu$ given by \eqref{eq_nu_metric} is a metric on $\mS(\calA_+)$.
\end{theorem}

Moreover, stabilizability is a robust property of the plant in this new
$\nu$-metric. In order to see this, we first introduce the notion of
stability margin for a pair comprising a plant and its controller.

\begin{definition}
Given $p,c \in \mF(\calA_+)$, the {\em stability margin} of the pair 
$(p,c)$ is defined by
$$
\mu_{p,c}=\left\{ \begin{array}{ll}
\|H(p,c)\|_{\infty}^{-1} &\textrm{if }p \textrm{ is stabilized by }c,\\
0 & \textrm{otherwise.}
\end{array}\right.
$$
\end{definition}

The number $\mu_{p,c}$ can be interpreted as a measure of the
performance of the closed loop system comprising $p$ and $c$: larger
values of $\mu_{p,c}$ correspond to better performance, with
$\mu_{p,c}>0$ if $c$ stabilizes $p$. 

The following was proved in \cite{BalSas}:

\begin{theorem}
\label{theorem_stab_margin_nu_metric}
If $p,p'\in \mS(\calA_+)$ and $c\in \mS(\calA_+)$, then $ \mu_{p',c}
\geq \mu_{p,c}-d_{\nu}(p,p')$.
\end{theorem}

The above result says that stabilizability is a robust property of the
plant, since if $c$ stabilizes $p$ with a stability margin
$\mu_{p,c}>m$, and $p'$ is another plant which is close to $p$ in the
sense that $d_\nu(p',p)\leq m$, then $c$ is also guaranteed to
stabilize $p'$.

We will now derive an alternative expression for the $\nu$-metric,
which is reminiscent of Georgiou's formula for the gap-metric from
\cite{Geo}.

\begin{proposition}
\label{prop_alt_exp_nu_metric}
If $p_1, p_2\in \mS(\calA_+)$, then 
 $$
d_\nu(p_1, p_2)= \displaystyle \inf_{\substack{q \in \textrm{\em inv } \calA ,\\
    W( q)=(0,0)}} \|G_1 -G_2 q\|_\infty.
$$
\end{proposition}
\begin{proof} Let $q \in \inv \calA $ and 
$W(q)=(0,0)$. We have 
\begin{eqnarray*}
\|G_1 -G_2 q\|_\infty &=& 
\left\| \left[ \begin{array}{cc} G_2^* \\ 
\widetilde{G}_2 \end{array} \right] (G_1-G_2 q) \right\|_\infty 
\quad \textrm{(as }  
\left[\begin{array}{cc} G_2 & \widetilde{G}_2^* \end{array}\right]
\left[\begin{array}{c} G_2^* \\\widetilde{G}_2 \end{array}\right]=I\textrm{)} \\
&=& 
\left\| \left[\begin{array}{c} G_2^* G_1-q\\ \widetilde{G}_2G_1 \end{array}\right]  \right\|_\infty 
\quad \textrm{(since } \widetilde{G}_2 G_2=0 \textrm{ and }  G_2^*G_2=I\textrm{)} \\
&\geq & \| \widetilde{G}_2G_1 \|_\infty .\phantom{\left[\begin{array}{c} A \\B \end{array}\right]}
\end{eqnarray*}
So if $G_2^* G_1 \in \inv \calA$ and $W (G_2^* G_1)=(0,0)$, then from
the above it follows that $ \|G_1 -G_2 q\|_\infty \geq \|
\widetilde{G}_2G_1 \|_\infty=d_\nu(p_1,p_2)$.  As the choice of $q$
above was arbitrary, we obtain
\begin{equation}
\label{eq_nu_metrix_alt_exp_ineq_1}
\inf_{\substack{q \in \inv \calA ,\\ 
W( q)=(0,0)}} \|G_1 -G_2 q\|_\infty \geq d_\nu(p_1,p_2).
\end{equation}
If we define $q_0:= G_2^* G_1 \in \calA$,
then $q_0\in \inv \calA$ and $W (q_0)=(0,0)$, and so 
\begin{eqnarray*}
  \inf_{\substack{q \in \inv \calA ,\\ 
      W(q)=(0,0)}} \|G_1 -G_2 q\|_\infty 
  &\leq&  \|G_1-G_2 q_0\|_\infty
  =
  \left\| \left[\begin{array}{c} G_2^* G_1-q_0\\ \widetilde{G}_2G_1 \end{array}\right]  \right\|_\infty 
  \\
  &=&\left\| \left[\begin{array}{c} 0\\ \widetilde{G}_2G_1 \end{array}\right]  \right\|_\infty 
  =
  \|\widetilde{G}_2G_1\|_\infty =d_\nu(p_1,p_2).
\end{eqnarray*}
From this and \eqref{eq_nu_metrix_alt_exp_ineq_1}, the claim in the
proposition follows for the case when $G_2^* G_1\in \inv \calA$ and $W
(G_2^* G_1)=(0,0)$.

Now let $q\in \inv \calA$ be such that $W (q)=(0,0)$ and
$\|G_1-G_2 q\|_\infty <1$. Using $G_1^* G_1=1$, we see that
$$
\|1-G_1^* G_2 q\|_\infty =\|G_1^*(G_1-G_2 q)\|_\infty 
\leq\|G_1^*\|_\infty \|G_1-G_2 q\|_\infty<1\cdot 1=1.
$$
So $ G_1^* G_2 q=1-(1-G_1^* G_2 q) $ is invertible as an element of
$\calA$. Consider the map $H:[0,1]\rightarrow \inv \calA $ given by $
H(t)= 1-t(1-G_1^* G_2 q)$, $t\in [0,1]$.  By the homotopic invariance
of the index, 
$$
(0,0)=W(1)=W (H(0))=W (H(1))=W(G_1^* G_2 q).
$$
As $W(q)=(0,0)$, we obtain that $W(G_1^* G_2)=(0,0)$. So we have shown
that if there is a $q\in \calA$ such that $q\in \inv \calA$, $W
(q)=(0,0)$ and $\|G_1-G_2 q\|_\infty <1$, then $G_1^* G_2 \in \inv
\calA$ and $W(G_1^* G_2)=(0,0)$. Thus if either $G_1^* G_2
\not\in \inv \calA$ or $G_1^* G_2 \in \inv \calA$ but $W(G_1^*
G_2)\neq (0,0)$, then for all $q\in \calA$ such that $q\in \inv
\calA$, $W (q)=(0,0)$, we have that $\|G_1-G_2 q\|_\infty \geq 1$, and
so
$$
\inf_{\substack{q \in \inv \calA ,\\ 
W(q)=(0,0)}} \|G_1 -G_2 q\|_\infty \geq 1=d_\nu(p_1,p_2).
$$
Also, with $q_n:=\displaystyle \frac{1}{n} I$, $q_n\in \inv \calA$ and
$W ( q_n)=(0,0)$. We have
$$
\|G_1-G_2 q_n\|_\infty \leq \|G_1\|_\infty + \|G_2\|_\infty
\|q_n\|_\infty\leq 1+1\cdot \frac{1}{n}.
$$
Hence
$$
\inf_{\substack{q \in \inv \calA ,\\
    W(q)=(0,0)}} \|G_1 -G_2 q\|_\infty \leq \inf_n \|G_1 -G_2
q_n\|_\infty\leq \inf_n \left(1+\frac{1}{n}\right)=1=d_\nu(p_1,p_2).
$$
Consequently, 
 $
\displaystyle \inf_{\substack{q \in \inv \calA ,\\ 
W(q)=(0,0)}} \|G_1 -G_2 q\|_\infty=1=d_\nu(p_1,p_2)$. 
\end{proof}

\section{The gap-metric}
\label{section_gap_metric}

In this section we will recall the gap-metric topology for unstable
plants over the ring $\calA_+$. We will also prove a few technical
lemmas which will be used in the next section in order to prove our
main result.

\begin{definition}[Graph of a system] 
  For $p\in \mS(\calA_+)$, with the normalized coprime factorization $
  p= n d^{-1}$, we define the {\em graph of} $p$, denoted by $\calG$,
  to be the following subspace of the Hardy space
  $H^2(\mC^2)$:\label{pageref_calG}
$$
\calG=G H^2=\left\{\left[ \begin{array}{cc} n\varphi \\ d
      \varphi\end{array} \right]: \varphi \in H^2\right\}.
$$
\end{definition}

Using the fact that there exist $x,y\in \calA_+$ such that $xn+yd=1$,
it is easy to see that the graph $\calG$ is a {\em closed} subspace of
$H^2\times H^2$. We denote the orthogonal projection from $H^2\times
H^2$ onto $\calG$ by $P_{\calG}$.\label{pageref_P_calG}

\begin{definition}[Gap-metric $d_g$]
\label{def_graph_metric}
For $p_1, p_2 \in \mS(\calA_+)$, with the normalized coprime
factorizations $p_1= n_{1} d_{1}^{-1}$ and $p_2= n_{2} d_{2}^{-1}$, we
define \label{pageref_d_g}
\begin{equation}
\label{eq_graph_metric}
d_{g} (p_1,p_2 ):=
\|P_{\calG_1}-P_{\calG_2}\|_{\calL(H^2\times H^2)}.
\end{equation}
\end{definition}

We will need a few technical results on the gap-metric $d_g$. For a
self-contained account of these results, we refer the reader to
\cite{Par}. It can be checked that $d_g$ given by
\eqref{eq_graph_metric} is well-defined.  Since the gap-metric is a
metric on the set of closed subspaces of a Hilbert space, it follows
that $d_g$ given by \eqref{eq_graph_metric} is a metric on
$\mS(\calA_+)$.

For $p_1,p_2 \in \mS(\calA_+)$, $ d_g(p_1,p_2)=\max\{
\vec{\delta}(p_1,p_2),\vec{\delta}(p_2,p_1)\}$, where $
\vec{\delta}(\cdot,\cdot)$ denotes the {\em directed gap}, defined
by\label{pageref_vec_delta}
$$
\vec{\delta}(p_1,p_2):= \|(I-P_{\calG_2})P_{\calG_1}\|_{\calL(H^2\times H^2)}.
$$
If $d_g(p_1,p_2)<1$, then $
d_g(p_1,p_2)=\vec{\delta}(p_1,p_2)=\vec{\delta}(p_2,p_1)$
\cite[Prop.~3, p.675]{GeoSmi}. In \cite{Geo}, it was shown that
$$
d_g(p_1,p_2)=\max\Big\{\inf_{q\in H^\infty} \|G_1-G_2 q\|_\infty\;,\;
  \inf_{q\in H^\infty} \|G_2-G_1 q\|_\infty\Big\}.
$$
For $p_1,p_2\in \mS(\calA_+)$, the infimums above can be taken over
$\calA_+$ instead of $H^\infty$, and this follows from
\cite[Theorem~11.3.3]{Mik}.

\begin{lemma}
\label{lemma_Kalle}
If $p_1,p_2\in \mS(\calA_+)$, then 
$$
\inf_{q\in H^\infty} \|G_1-G_2 q\|_\infty = \inf_{q\in \calA_+}
\|G_1-G_2 q\|_\infty.
$$
\end{lemma}
\begin{proof} Clearly
 $
m:= \displaystyle\inf_{q\in H^\infty} \|G_1-G_2 q\|_\infty \leq \displaystyle\inf_{q\in \calA_+}
\|G_1-G_2 q\|_\infty=:M.$ 
Define
$$
V=\left[\begin{array}{cccc} G_2 & G_1 \\ 0 & 1 \end{array}\right],
\quad W:=V^\star \left[\begin{array}{cccc} I & 0 \\0 & -M^2
  \end{array}\right] V.
$$
(For $X\in (H^\infty)^{p\times m}$, $X^\star\in
(L^\infty)^{m\times p}$ is defined by $X^\star(iy)=(X(iy))^*$, $y\in
\mR$.) \label{pageref_star_again}  Suppose that $m<M$. Then there exists a $q\in H^\infty$ such
that $\|G_1-G_2 q\|_\infty <M$. Now we apply \cite[Theorem~11.3.3,
p.654]{Mik} to conclude that the $q$ can in fact be chosen in
$\calA_+$. For this, a few technical assumptions have to be verified first, and
we give these details in the following paragraph for the interested
reader.

(First of all, the Standing Hypothesis \cite[11.0.1, p.611]{Mik} is
satisfied, since $\calA_+$ does satisfy \cite[Hypothesis 8.4.7.,
p.384]{Mik}, by \cite[Theorem 8.4.9($\beta$), p.385]{Mik}. Secondly,
the Standing Hypothesis \cite[11.3.1, p.654]{Mik} is satisfied, since
$G_2^* G_2=1$. Actually, there are two extraneous assumptions in
11.3.1, but neither is used in the part of the proofs required here,
and these extraneous assumptions are anyway satisfied in our case.
Now as the Assumption (FI1$\frac{1}{2}$s) of \cite[Theorem~11.3.3,
p.654]{Mik} holds, also (FI13s) holds. By the last sentence of
\cite[Theorem~11.3.6, p.659]{Mik}, as $W$ has entries from $\calA_+$,
there exists a $q\in \calA_+$ such that $ \|G_1-G_2 q\|_\infty<M$.)

Consequently, $m=M$.
\end{proof}

We use the notation $P_{\calG_1}|_{\calG_2}$ to mean the restriction
of $P_{\calG_1}$ to $\calG_2$, namely, the operator from $\calG_2$ to
$\calG_1$, given by \label{pageref_P_calG_1_calG_2}
$$
P_{\calG_1}|_{\calG_2} g_2= P_{\calG_1} g_2 ,\quad g_2\in \calG_2.
$$
Then $ \ker (P_{\calG_1}|_{\calG_2})= \{g_2\in \calG_2: P_{\calG_1}
g_2 =0\} =\calG_2\cap (\ker P_{\calG_1})=\calG_2\cap \calG_1^\perp$.
Also, for $g_1\in \calG_1$ and $g_2\in\calG_2$, we have
\begin{eqnarray*}
\langle P_{\calG_1}|_{\calG_2} g_2 ,g_1\rangle_{\calG_1} 
&=& \langle
P_{\calG_1} g_2 ,g_1\rangle_{\calG_1} = \langle g_2
,g_1\rangle_{H^2(\mC^2)}\\
& =& \langle g_2 ,P_{\calG_2}
g_1\rangle_{H^2(\mC^2)} = \langle g_2 ,P_{\calG_2}
g_1\rangle_{\calG_2} \\
&=& \langle g_2 ,P_{\calG_2} |_{\calG_1}
g_1\rangle_{\calG_2},
\end{eqnarray*}
and so $( P_{\calG_1}|_{\calG_2})^*= P_{\calG_2} |_{\calG_1}$. Thus $
\ker ((P_{\calG_1}|_{\calG_2})^*)=\ker (P_{\calG_2}
|_{\calG_1})=\calG_1\cap \calG_2^\perp$. So if
$P_{\calG_1}|_{\calG_2}$ is a Fredholm operator \cite[\S
2.5.1,p.218]{NikA}, then its Fredholm index is given by $ \dim
(\calG_2\cap \calG_1^\perp)-\dim (\calG_1\cap \calG_2^\perp)$.

We will use the following result from \cite[p.201]{Nik}.

\begin{lemma}[Lemma on Closed Subspaces]
\label{lemma_cl_sub}
Let $H$ be a Hilbert space and let $U,V$ be subspaces of $H$. Then the
following are equivalent:
\begin{itemize}
\item[(S1)] $U\cap V^\perp=\{0\}$.
\item[(S2)] Closure of $P_U V$ is $U$.
\end{itemize}
Also, the following are equivalent:
\begin{itemize}
\item[(S3)] $P_U V=U$ and $V\cap U^\perp=\{0\}$.
\item[(S4)] $\|(I-P_V)P_U\|<1$ and $\|(I-P_U)P_V\|<1$. 
\end{itemize}
\end{lemma}

\begin{lemma}
\label{lemma_fred_1}
Let $p_1,p_2\in \mS(\calA_+)$. Then $d_g(p_1,p_2)<1$ if and only if
the following three conditions hold:
\begin{enumerate}
\item $P_{\calG_1}|_{\calG_2}$ is Fredholm, 
\item $\calG_1 \cap \calG_2^\perp=\{0\}$, and 
\item $\calG_2 \cap \calG_1^\perp=\{0\}$.
\end{enumerate}
\end{lemma}
\begin{proof} (Only if) As $P_{\calG_1}|\calG_2$ is Fredholm, its
  range is closed, that is, $P_{\calG_1}\calG_2$ is a closed subspace.
  Hence from the equivalence of (S1) with (S2) in
  Lemma~\ref{lemma_cl_sub} above, we have that the closure of
  $P_{\calG_1} \calG_2$, which is the same as $P_{\calG_1} \calG_2$,
  is equal to $\calG_1$. Now from the equivalence of (S3) with (S4) in
  Lemma~\ref{lemma_cl_sub}, we obtain that
  $\vec{\delta}(p_1,p_2)=\|(I-P_{\calG_2})P_{\calG_1}\|<1$ and
  $\vec{\delta}(p_2,p_1)=\|(I-P_{\calG_1})P_{\calG_2}\|<1$. Hence
  $d_g(p_1,p_2)<1$.

\medskip

\noindent (If) As $\vec{\delta}(p_1,p_2)=\|(I-P_{\calG_2})
P_{\calG_1}\|<1$ and
$\vec{\delta}(p_2,p_1)=\|(I-P_{\calG_1})P_{\calG_2}\|<1$, by the
equivalence of (S3) with (S4) in Lemma~\ref{lemma_cl_sub}, we obtain
 $P_{\calG_1} \calG_2=\calG_1$, and so the range of
$P_{\calG_1}|_{\calG_2}$ is closed. Moreover, $\calG_2 \cap
\calG_1^\perp=\{0\}$. By interchanging the roles of $p_1$ and $p_2$,
we also get that $\calG_1 \cap \calG_2^\perp=\{0\}$.
\end{proof}

The following is easy to check.

\begin{lemma}
\label{lemma_li}
Let $H_1, H_2$ be Hilbert spaces and $T\in \calL(H_1,H_2)$, $S\in
\calL(H_2,H_1)$ be such that $ST=I$. Suppose that $U$ is a subspace of
$H_1$. Then we have that $TU$ is closed if and only if $U$ is closed.
\end{lemma}
\begin{proof} (If) Since $T$ is left invertible, $ \|x\|=\|STx\|\leq
  \|S\|\|Tx\|$ ($x\in H_1$).  Suppose $(y_n)=(Tx_n)$ ($x_n\in U$) is a
  sequence that converges in $H_2$. Thus $ \|y_n-y_m\|\geq
  \frac{1}{\|S\|} \|x_n-x_m\|$, showing that $(x_n)$ must converge to
  some $x\in H_1$. As $U$ is closed, $x\in U$. Thus
  $y_n=Tx_n\rightarrow Tx\in TU$. Hence $TU $ is closed.

\medskip 

\noindent (Only if) Now suppose that $TU$ is closed. If $(x_n)$ is a
sequence in $U$ that converges to $x$ in $H_1$, then clearly $T
x_n\rightarrow Tx$. But $TU$ is closed, and so $Tx\in TU$. Hence
$Tx=Tx'$ for some $x'\in U$. Operating by $S$, we have
$x=STx=STx'=x'$, and so $x=x'\in U$. Thus $U$ is closed. 
\end{proof}

For $X\in (L^{\infty})^{p\times m}$, $T_X$ denotes the {\em Toeplitz
  operator} from $(H^2)^m$ to $(H^2)^p$, given by $ T_X
\varphi=\Pi_{(H^2)^p}(X\varphi)$ ($\varphi \in (H^2)^m$), where
$X\varphi$ is considered as an element of $(L^2)^p$ and
$\Pi_{(H^2)^p}$ denotes the canonical orthogonal projection from
$(L^2)^p$ onto $(H^2)^p$. \label{pageref_T_X}

\begin{lemma}
\label{lemma_fred_2}
  Let $p_1,p_2\in \mS(\calA_+)$. Then $P_{\calG_1}|_{\calG_2}$ is
  Fredholm if and only if $T_{G_1^* G_2}$ is Fredholm. Moreover, their
  Fredholm indices coincide.
\end{lemma}
\begin{proof} First of all, we note that $T_{G_1^*
    G_2}=T_{G_1^*}T_{G_2}$ (since $G_2$ has $H^\infty$ entries). Also,
  it can be checked that for a matrix $X$ with $L^\infty$ entries $
  (T_X)^*=T_{X^*}$.  Thus $(T_{G_1^* G_2})^*= T_{G_2^* G_1}$.

As $T_{G_1}$ is an isometry, it follows that the orthogonal projection
onto the range of $T_{G_1}$, namely the subspace $\calG_1$, is given
by $T_{G_1} (T_{G_1})^*= T_{G_1} T_{G_1^*}$. Indeed, with $P:=
T_{G_1} T_{G_1^*}$, and using $G_1^* G_1=1$, we can check that
$P^2=P$, that $P^*=P$ and that $P$ maps onto the range of $T_{G_1}$:
$$
 \textrm{ran } (T_{G_1} T_{G_1^*} )\subset  \textrm{ran } T_{G_1}
=\textrm{ran } (T_{G_1}T_{G_1^*} T_{G_1})\subset \textrm{ran } (T_{G_1}T_{G_1^*}).
$$
We have that 
\begin{eqnarray*}
\ker (T_{G_1^*}T_{G_2})&=& \{\varphi \in H^2 : T_{G_1^*}T_{G_2} \varphi=0\}
\\
&=&  \{\varphi \in H^2 : T_{G_1}T_{G_1^*}T_{G_2} \varphi=0\} \quad 
\textrm{(since }\left[\begin{array}{cc} x_1 & y_1  \end{array}\right] G_1=1\textrm{)}
\\
&=& \{\varphi \in H^2 : P_{\calG_1} T_{G_2} \varphi=0\}
=  \{\varphi \in H^2 : T_{G_2}\varphi \in \calG_1^{\perp}\}.
\end{eqnarray*}
Consider the map $\iota:\ker (T_{G_1^*}T_{G_2})\rightarrow
\calG_1^\perp \cap \calG_2$ defined by $\iota
(\varphi)=T_{G_2}\varphi$ for $\varphi \in \ker (T_{G_1^*}T_{G_2})$.
From the above calculation, we see that $ \iota$ is onto. Also, since
$ \left[\begin{array}{cc} x_2 & y_2 \end{array}\right] G_2=1$ it
follows that $\iota$ is one-to-one. So $\iota$ is invertible.

The above shows that in case that $P_{\calG_1}|_{\calG_2}$ and
$T_{G_1^* G_2}$ are both Fredholm operators, their Fredholm indices
will coincide.

In light of the above, we just need to show that the range of $
P_{\calG_1}|_{\calG_2}$ is closed if and only if the range of $T_{G_1^*
  G_2}$ is closed.  The range of $P_{\calG_1}|_{\calG_2}$ is
$$
P_{\calG_1} \calG_2=P_{\calG_1} \textrm{ran }T_{G_2}=T_{G_1}T_{G_1^*} \textrm{ran }T_{G_2}
= T_{G_1}\textrm{ran }T_{G_1^* G_2}.
$$
Since $G_1$ has a left inverse $\left[\begin{array}{cc} x_1 &
    y_1\end{array}\right]\in \calA_+^2$, it follows that $T_{G_1}$ is
left-invertible. By Lemma~\ref{lemma_li}, the range of
$P_{\calG_1}|_{\calG_2}$ is closed if and only if the range of
$\textrm{ran }T_{G_1^* G_2}$ is closed.
\end{proof}

We will need the following result, which follows from \cite[Thm.~3,
p.150]{Dou}. Here $C_0$ denotes the set of continuous functions on
$\mR$ that vanish at $\pm\infty$. \label{pageref_C_0}

\begin{proposition}
\label{prop_douglas}
Let $F=f+g$, where $f\in AP$ and $g\in C_0$ be such that $T_F$ is Fredholm. 
Then the following hold:
\begin{enumerate}
\item $T_f$ is invertible.
\item $F\in \textrm{\em inv } (AP+C_0)$.
\item The Fredholm index of $T_F$ is the winding number of $1+f^{-1}g$. 
\end{enumerate}
\end{proposition}
\begin{proof} Since $T_F$ is invertible modulo the compacts, it is
  invertible modulo any bigger ideal which we can take to be the
  kernel of the symbol map from the Toeplitz $C^*$-algebra
  $\calT(AP+C_0)$ (generated by $T_\varphi$ for
  $\varphi \in AP+C_0$) to $AP+C_0$.  Consequently, there must exist
  $\epsilon > 0$ such that $|f + g| > \epsilon$ on all of $\mR$.

  Since $g$ is in $C_0$, it follows that by choosing
  $a$ large enough we can assume that $|g(x)| < \epsilon/2$ for $x >
  a$ and hence $|f(x)| > \epsilon/2$ for $x > a$.  Since $f\in AP$, it
  follows that $f(x)\neq 0$ for all $x\in \mR$. Therefore $f$ is
  invertible in $AP$.  Moreover, using \cite[Theorem~3, p.150]{Dou},
  one knows that its generalized index is $(0,n)$ for some integer $n$
  and hence the average winding number of $f$ is zero.  Thus $T_f$ is
  invertible \cite[Theorem~11, p.25]{DouA}.

  Again using \cite[Theorem~3, p.150]{Dou}, one can see that the
  generalized index of $T_F$ equals the sum of the generalized indices
  of $T_f$ and $T_{1 + f^{-1} g}$.  But the generalized index of $T_f$
  is $(0,0)$ which completes the proof.
\end{proof}

\begin{proposition}
\label{prop_alt_exp_gap_metric}
If $p_1,p_2\in \mS(\calA_+)$, then 
$$
d_g(p_1,p_2)=\inf_{q\in \textrm{\em inv } \calA_+} \|G_1-G_2 q\|_\infty
$$
\end{proposition}
\begin{proof} $\underline{1}^\circ$ Consider first the case when
  $d_g(p_1,p_2)<1$. From Lemma~\ref{lemma_fred_1}, it follows that
  $P_{\calG_1}|_{\calG_2}$ is Fredholm, $\calG_1\cap \calG_2^\perp
  =\{0\}$ and $\calG_2\cap \calG_1^\perp =\{0\}$. Furthermore, the
  Fredholm index of $P_{\calG_1}|_{\calG_2}$ is $0$. By
  Lemma~\ref{lemma_fred_2}, $T_{G_1^* G_2}$ is Fredholm, with Fredholm
  index $0$ too. From Proposition~\ref{prop_douglas}, it follows that
  $G_1^* G_2$ is invertible as an element of $AP+C_0$.  Thus it is
  also invertible as an element of $\calA$. Also, $W (G_1^*
  G_2)=(0,0)$. Now suppose that there is a $q_0\in \calA_+$ such that
  $ \| G_1-G_2 q_0\|_\infty <1$.  Then $\|I-G_1^* G_2 q_0\|_\infty <1$
  and so $G_1^* G_2 q_0=1-(1-G_1^* G_2 q_0)$ is invertible in $\calA$.
  Hence $G_1^* G_2 q_0 \in \inv \calA$. In particular, $q_0\in \inv
  \calA$. Consider the map $H:[0,1]\rightarrow \inv \calA$ given by $
  H(t)= 1-t(1-G_1^* G_2 q_0)$, $t \in [0,1]$. By the homotopic
  invariance,
  $$
  (0,0)=W(1)=W(H(0))=W(H(1))=W(G_1^* G_2 q_0).
  $$ 
  Since $W(G_1^* G_2)=(0,0)$, it follows that $W(q_0)=(0,0)$.
  Thus by (I4), we obtain that $q_0 \in \inv \calA_+$. Consequently,
  $$ 
  1>d_g(p_1,p_2)=\vec{\delta}(p_1,p_2)=\inf_{q\in \calA_+}\|G_1-G_2 q\|_\infty
  = \inf_{q\in \inv \calA_+}\|G_1-G_2 q\|_\infty.
  $$
 
  \noindent $\underline{2}^\circ$ Now suppose that $d_g(p_1,p_2)=1$,
  but that $\vec{\delta}(p_1,p_2)<1$. Since we have $
  \vec{\delta}(p_1,p_2)=\|(I-P_{\calG_2})P_{\calG_1}\|$, we obtain
  $\calG_1 \cap \calG_2^\perp=\{0\}$. For otherwise, if $0\neq v\in
  \calG_1 \cap \calG_2^\perp$, then we have
  $(I-P_{\calG_2})P_{\calG_1}v=v$, and so we would obtain that 
   $
  \|(I-P_{\calG_2})P_{\calG_1}\|\geq \|(I-P_{\calG_2})P_{\calG_1}v\|/\|v\|=1,
  $ 
  a contradiction. From Lemma~\ref{lemma_fred_1}, it now follows that
  either $\calG_2 \cap \calG_1^\perp\neq \{0\}$ or
  $P_{\calG_1}|_{\calG_2}$ is not Fredholm. 

  Suppose first that $P_{\calG_1}|_{\calG_2}$ is Fredholm. Then we
  must have $\calG_2 \cap \calG_1^\perp\neq \{0\}$. This gives that
  the Fredholm index of $P_{\calG_1}|_{\calG_2}$, namely
$$
\dim (\calG_2\cap \calG_1^\perp)-\dim (\calG_1\cap \calG_2^\perp)=\dim
(\calG_2\cap \calG_1^\perp)-0 =\dim (\calG_2\cap \calG_1^\perp),
 $$
 is nonzero. By Lemma~\ref{lemma_fred_2}, $T_{G_1^* G_2}$ is
 Fredholm, with Fredholm index nonzero too. It now follows from
 Proposition~\ref{prop_douglas}, that $W(G_1^* G_2)=(\ast,n)$ with the
 integer $n\neq 0$. By the definition of $d_\nu$, $d_\nu(p_1,p_2)=1$.

 Next assume that $P_{\calG_1}|_{\calG_2}$ is not Fredholm. Then
 Lemma~\ref{lemma_fred_2} gives that $T_{G_1^* G_2}$ is not Fredholm
 either. Now if $G_1^*G_2$ is not invertible in $\calA$, then we have
 $d_\nu(p_1,p_2)=1$ by definition. On the other hand, if $G_1^*G_2\in
 \inv \calA$ and $W(G_1^*G_2)=(0,0)$, it follows from \cite[Proposition~6.3, p.27]{DouA} 
 that $T_{G_1^*G_2}$ is invertible, a contradiction. Thus
 $W(G_1^*G_2)=(0,0)$, and so $d_\nu(p_1,p_2)=1$ in this case as well. 

 Now that we have obtained $d_\nu(p_1,p_2)=1$, it follows that there
 is no $q\in \inv \calA_+$ such that $\|G_1-G_2q\|_\infty<1$. In other
 words, for each $q\in \calA_+$, $\|G_1-G_2q\|_\infty\geq 1$. Also, 
 $
\|G_1-G_2q\|_\infty\leq \|G_1\|_\infty +\|G_2\|_\infty
\|q\|_\infty\leq 1+1\cdot \|q\|_\infty$, 
and by taking $q_1=\frac{1}{n}I\in \inv \calA_+$, we obtain
$$
\inf_{q\in \inv\calA_+}  \|G_1-G_2 q\|_\infty \leq \inf_n \left(1+\frac{1}{n}\right)=1.
$$
Consequently, $\displaystyle 
\inf_{q\in \inv\calA_+}  \|G_1-G_2 q\|_\infty=1=d_g(p_1,p_2)$.

\medskip 

\noindent $\underline{3}^\circ$ Now suppose that
$d_g(p_1,p_2)=1=\vec{\delta}(p_1,p_2)=\vec{\delta}(p_2,p_1)$. 
We have 
$$
\inf_{q\in \inv\calA_+} \|G_1-G_2 q\|_\infty \leq \inf_n \Big\|G_1-G_2
\frac{1}{n}I \Big\|_\infty\leq \inf_n \left(1+\frac{1}{n}\right)=1.
$$
Also, $ 1=\vec{\delta}(p_1,p_2)=\displaystyle \inf_{q\in \calA_+}
\|G_1-G_2 q\|_\infty\leq \displaystyle \inf_{q\in \inv \calA_+}
\|G_1-G_2 q\|_\infty$.  Thus  
$$
\displaystyle 
\inf_{q\in \inv \calA_+}  \|G_1-G_2 q\|_\infty=1=d_g(p_1,p_2).
$$

This completes the proof.
\end{proof}

\section{Equivalence of the $\nu$-metric and the gap-metric}
\label{section_equivalence}

\begin{proof}[Proof of Theorem~~\ref{main_theorem}] We will show the
  following for $p_1,p_2\in \mS(\calA_+)$:
\begin{equation}
\label{eq_equiv_inqs}
d_g(p_1,p_2)\mu_{\textrm{opt}}(p_1)\leq d_\nu(p_1,p_2) \leq d_g(p_1,p_2),
\end{equation}
where $\mu_{\textrm{opt}}(p_1):=\displaystyle \sup_{c} \mu_{p_1,c}$. 

This will prove the fact that the topologies induced by metrics $d_g$
and $d_\nu$ on the set $\mS(\calA_+)$ are identical.

The second inequality in \eqref{eq_equiv_inqs} is an immediate
consequence of the Propositions~\ref{prop_alt_exp_nu_metric} and
\ref{prop_alt_exp_nu_metric}. Indeed, we have
\begin{eqnarray*}
  d_\nu(p_1,p_2)&=& \inf_{\substack{q \in \inv \calA ,\\
      W( q)=(0,0)}} \|G_1 -G_2 q\|_\infty\\
  &\leq & \inf_{\substack{q \in \calA_+ \cap (\inv \calA) ,\\
      W( q)=(0,0)}} \|G_1 -G_2 q\|_\infty\\
  &=&  \inf_{q \in \inv \calA_+} \|G_1 -G_2 q\|_\infty \quad 
   \textrm{(using (I4))}\phantom{\inf_{\substack{q\\ W}} \|G\|}\\
  &=& d_g(p_1,p_2). \phantom{\inf_{\substack{q\\W}} \|G\|}
\end{eqnarray*}
Now we will show the first inequality in \eqref{eq_equiv_inqs}. This
inequality is trivially satisfied if $d_\nu(p_1,p_2)\geq
\mu_{\textrm{opt}}(p_1)$, since $d_g(p_1,p_2)\leq 1$. So we will only
consider the case when $ d_\nu (p_1,p_2)<\mu_{\textrm{opt}}(p_1)$.
Thus we can choose a $c$ that stabilizes both $p_1$ and $p_2$. (Since
the above inequality shows that there exists a $c_0$ stabilizing $p_1$
such that $d_\nu(p_1,p_2)<\mu_{p_1,c_0}$. But by
Theorem~\ref{theorem_stab_margin_nu_metric}, it follows that $c_0$
also stabilizes $p_2$.)  If we now define $ q_0:= ( \widetilde{K}_0
G_1)^{-1}\widetilde{K}_0 G_2$, then we have $G_2- G_1 q_0 = G_2- G_1 (
\widetilde{K}_0 G_1)^{-1}\widetilde{K}_0G_2 = (I-G_1 ( \widetilde{K}_0
G_1)^{-1}\widetilde{K}_0)G_2 $. Also
$$
I-\left[\begin{array}{cc} p_1 \\ 1 \end{array}\right] (1-c_0 p_1)^{-1} \left[\begin{array}{cc}-c_0 & 1 \end{array}\right]
 = 
\left[\begin{array}{cc} 1 \\ c_0 \end{array}\right] (1-p_1c_0)^{-1} 
 \left[\begin{array}{cc}1 & -p_1 \end{array}\right].
$$
that is, $ I-G_1 ( \widetilde{K}_0 G_1)^{-1}\widetilde{K}_0=K_0
(\widetilde{G}_1 K_0)^{-1}\widetilde{G}_1$.  Thus
$$
G_2- G_1 q_0 =K_0 (\widetilde{G}_1 K_0)^{-1}\widetilde{G}_1G_2.
$$
Then we use $\|K_0\|\leq 1$ (since $K_0^* K_0=1$) to obtain
\begin{eqnarray*}
\|G_2 -G_1 q_0\|_\infty &=& \|K_0 (\widetilde{G}_1 K_0)^{-1}\widetilde{G}_1   G_2\|_\infty\\
&\leq & \|K_0\|_\infty\|(\widetilde{G}_1 K_0)^{-1}\widetilde{G}_1   G_2\|_\infty\\
&\leq & 1 \cdot \|(\widetilde{G}_1 K_0)^{-1}\widetilde{G}_1   G_2\|_\infty\\
&\leq &  \|(\widetilde{G}_1 K_0)^{-1}\|_\infty\|\widetilde{G}_1   G_2\|_\infty.
\end{eqnarray*}
As for each $c$, $\mu_{p_1,c}\leq 1$, we have
$\mu_{\textrm{opt}}(p_1)\leq 1$. So $d_\nu
(p_1,p_2)<\mu_{\textrm{opt}}(p_1)\leq 1$, and we obtain 
$d_\nu(p_1,p_2)=\|\widetilde{G}_1 G_2\|_\infty$.

From \cite[Propositions~4.2,4.5]{BalSas}, $ \|(\widetilde{G}_1
K_0)^{-1}\|_\infty=1/\mu_{c_0,p_1}=1/\mu_{p_1,c_0}$.  So 
$$
\|G_2 -G_1 q_0\|_\infty\leq  \|(\widetilde{G}_1 K_0)^{-1}\|_\infty\|\widetilde{G}_1   G_2\|_\infty
\leq \frac{d_\nu(p_1,p_2)}{\mu_{p_1,c_0}}.
$$
Thus 
$$
d_g(p_1,p_2) = \displaystyle \inf_{q\in \inv \calA} \|G_1-G_2 q\|_\infty \leq
\|G_1-G_2 q_0\|\leq d_\nu(p_1,p_2)/\mu_{p_1,c_0}.
$$ 
This inequality holds for any $c_0$ that stabilizes $p_1$ for which
there holds $d_\nu(p_1,p_2)<\mu_{p_1,c_0}$. We can choose a sequence
$(c_{0,n})$ such $\mu_{p_1,c_{0,n}} \rightarrow
\mu_{\textrm{opt}}(p_1)$ as $n\rightarrow \infty$. Thus $ d_g(p_1,p_2)
\leq d_\nu(p_1,p_2)/\mu_{\textrm{opt}}(p_1)$.  This completes the
proof the first inequality in \eqref{eq_equiv_inqs}.
\end{proof}

The question of whether our main result, Theorem~\ref{main_theorem}
remains true for systems with multiple inputs and multiple outputs (as
opposed to just {\em scalar} inputs and outputs) is open.  A key
technical difficulty is the validity of the analogue of
Proposition~\ref{prop_douglas} for matricial data.

\medskip

  \noindent {\bf Acknowledgements:} The author thanks Professor Ronald
  Douglas for providing clarifications in \cite[Theorem~3]{Dou}, which
  are presented in Proposition~\ref{prop_douglas} of this article.
  Part (\ref{Kalle_remark}) of the Remark~\ref{remark} and the proof
  of Lemma~\ref{lemma_Kalle} are due to Kalle Mikkola, and the author
  thanks him for these.  A useful discussion with Professor
  Joseph Ball is also gratefully acknowledged.

\end{document}